\documentclass[a4paper,12pt]{amsart}
\usepackage{amsmath}
\usepackage{amssymb}
\usepackage{xcolor}
\usepackage{amsthm}
\usepackage{amsfonts}
\usepackage[francais, english]{babel}

\DeclareMathSymbol{\subsetneqq}{\mathbin}{AMSb}{36}

\newcommand{\R}{\mathbb{R}}

\newcommand{\N}{\mathbb{N}}

\newcommand{\C}{\mathbb{C}}

\newcommand{\beq}{\begin{eqnarray}}
\newcommand{\eeq}{\end{eqnarray}}
\newcommand{\bq}{\begin{equation}}
\newcommand{\eq}{\end{equation}}
\newcommand{\beqn}{\begin{eqnarray*}}
\newcommand{\eeqn}{\end{eqnarray*}}
\newcommand{\bex}{\begin{exo}}
\newcommand{\eex}{\end{exo}}
\newcommand{\ben}{\begin{enumerate}}
\newcommand{\een}{\end{enumerate}}

\newtheorem{th1}{{\bf Theorem}}[section]
\newtheorem{thm}[th1]{{\bf Theorem}}
\newtheorem{lem}[th1]{{\bf Lemma}}
\newtheorem{prop}[th1]{{\bf Proposition}}

\newtheorem{defi}[th1]{\bf Definition}

\author[ T. Saanouni]{ T. Saanouni}
\address{University Tunis El Manar,
Faculty of Sciences of Tunis, LR03ES04 partial differential equations and applications, 2092 Tunis, Tunisia.}
\email{\sl Tarek.saanouni@ipeiem.rnu.tn}


\title[coupled NLS]{On defocusing coupled nonlinear Schr\"odinger equations}

\date{\today}
\begin{document}
\begin{abstract}
The initial value problem for some defocusing coupled nonlinear Schr\"odinger equations is investigated. Global well-posedness and scattering are established.
\end{abstract}
\maketitle
\tableofcontents
\vspace{ 1\baselineskip}
\renewcommand{\theequation}{\thesection.\arabic{equation}}
\section{Introduction}
We consider the Cauchy problem for a defocusign Schr\"odinger system with power-type nonlinearities
\begin{equation}
\left\{
\begin{array}{ll}
i\dot u_j +\Delta u_j=\displaystyle\sum_{k=1}^{m}a_{jk}|u_k|^p|u_j|^{p-2}u_j ;\\
u_j(0,x)= \psi_{j}(x),
\label{S}
\end{array}
\right.
\end{equation}
where $u_j: \R \times \R^N \to \C$ for $j\in[1,m]$ and $a_{jk} =a_{kj}$ are positive real numbers.\\
The m-component coupled nonlinear Schr\"odinger system with power-type nonlinearities arises in many physical problems. This models physical systems in which the field has more than one component. For example, in optical fibers and waveguides, the propagating electric field has two components that are transverse to the direction of propagation.  When $m = 2$, this system also arises in the Hartree-Fock theory for a double condensate; i.e., a binary mixture of Bose-Einstein condensatesin in two different hyperfine states. Readers are referred to various other works \cite{Hasegawa,Zakharov} for the derivation and applications of this system.\\
A solution ${\bf u}:= (u_1,...,u_m)$ to \eqref{S} formally satisfies respectively conservation of the mass and the energy
\begin{gather*}
M(u_j):= \displaystyle\int_{\R^N}|u_j(x,t)|^2\,dx = M(\psi_{j});\\
E({\bf u}(t)):= \frac{1}{2}\displaystyle \sum_{j=1}^{m}\displaystyle\int_{\R^N}\Big(|\nabla u_j|^2+ \frac{1}{p}\displaystyle \sum_{k=1}^{m}a_{jk}\displaystyle  |u_j(x,t)|^p |u_k(x,t)|^p\Big)\,dx = E({\bf u}(0)).
\end{gather*}
Before going further, let us recall some historic facts about the one component case. The model case given by a pure power nonlinearity is of particular interest. The question of well-posedness in the energy space was widely investigated. We denote for $p>1$ the Schr\"odinger problem
$$(NLS)_p\quad i\dot  u+\Delta u\pm u|u|^{p-1}=0,\quad u:{\mathbb R}\times{\mathbb R}^N\rightarrow{\mathbb C}.$$
This equation satisfies a scaling invariance. Indeed, if $u$ is a solution to $(NLS)_p$ with data $u_0$, then
$ u_\lambda:=\lambda^{\frac2{p-1}}u(\lambda^2\, .\,,\lambda\, .\,)$
is a solution to $(NLS)_p$ with data $\lambda^{\frac2{p-1}}u_0(\lambda\,.\,).$
For $s_c:=\frac N2-\frac2{p-1}$, the space $\dot H^{s_c}$ whose norm is invariant under the dilatation $u\mapsto u_{\lambda}$ is relevant in this theory. When $s_c=1$ which is the energy critical case, the critical power is $p_c:=\frac{N+2}{N-2}$, $N\geq 3$. \\
Local well-posedness holds in the energy critical case \cite{Cas.F} and the local existence interval does not depend only on $\|u_0\|_{H^1}$. Then, an iteration of the local well-posedness theory fails to prove global existence. But using new ideas of Bourgain in \cite{J.B1,J.B2} and a new interaction Morawetz inequality \cite{Col.K} the energy critical case of $(NLS)_p$ is now completely resolved \cite{V,RV07}. Finite energy initial data $u_0$ evolve into global solution $u$ with finite spacetime size $\|u\|_{L_{t,x}^{\frac{2(2+N)}{N-2}}}<\infty$ and scatters.\\

In two space dimensions, similar results about global well-posedness and scattering of the Schr\"odinger equation with exponential nonlinearity exist \cite{T,T1,T2,T3}.\\

Intensive work has been done in the last few years, about coupled Schr\"odinger systems \cite{ntds,w,mz}. These works have been mainly on 2-systems or with small couplings. Moreover, most works treat the focusing case by considering the stationary associated problem \cite{AC2,xs,hs,AC3,AC4}. Despite the partial progress made so far, many difficult questions remain open and little is known about m-systems for $m\geq 3$.\\

It is the purpose of this manusrcipt to obtain global well-posedness and scattering of the nonlinear Schr\"odinger coupled nonlinear system \eqref{S}. This note extends \cite{ntds}, where local well-posedness was claimed without giving a proof.\\

The rest of the paper is organized as follows. The next section contains the main results and some technical tools needed in the sequel. Sections three and four are devoted to proving well-posedness of the Sch\"odinger system \eqref{S}. In section five, scattering is established. Finally, we give a proof of Morawetz estimate in appendix.\\
\\
We define the product space
$$H:=H^1(\R^N)\times...\times H^1(\R^N) =[H^1(\R^N)]^m$$
where $H^1(\R^N)$ is the usual Sobolev space endowed with the complete norm
$$ \|u\|_{H^1(\R^N)} := \Big(\|u\|_{L^2(\R^N)}^2 + \|\nabla u\|_{L^2(\R^N)}^2\Big)^\frac12$$
We denote the real numbers
 $$p_*:=1+\frac4N\quad\mbox{ and }\quad p^*:=\left\{
\begin{array}{ll}
\frac{N}{N-2}\quad\mbox{if}\quad N>2;\\
\infty\quad\mbox{if}\quad N=2.
\end{array}
\right.$$
We mention that $C$ will denote a
constant which may vary
from line to line and if $A$ and $B$ are nonnegative real numbers, $A\lesssim B$  means that $A\leq CB$. For $1\leq r\leq\infty$ and $(s,T)\in [1,\infty)\times (0,\infty)$, we denote the Lebesgue space $L^r:=L^r({\mathbb R}^N)$ with the usual norm $\|\,.\,\|_r:=\|\,.\,\|_{L^r}$, $\|\,.\,\|:=\|\,.\,\|_2$ and
$$\|u\|_{L_T^s(L^r)}:=\Big(\int_{0}^{T}\|u(t)\|_r^s\,dt\Big)^{\frac{1}{s}},\quad \|u\|_{L^s(L^r)}:=\Big(\int_{0}^{+\infty}\|u(t)\|_r^s\,dt\Big)^{\frac{1}{s}}.$$
For simplicity, we denote the usual Sobolev Space $W^{s,p}:=W^{s,p}({\mathbb R}^N)$ and $H^s:=W^{s,2}$. If $X$ is an abstract space $C_T(X):=C([0,T],X)$ stands for the set of continuous functions valued in $X$, moreover for an eventual solution to \eqref{S}, we denote $T^*>0$ it's lifespan.

\section{Main results and background}
In what follows, we give the main results and some estimates needed in the sequel.
\subsection{Main results}
First, local well-posedness of the Schr\"odinger problem \eqref{S} is claimed.
\begin{thm}\label{existence}
Let $2\leq N\leq 4$ and $ \Psi \in H$. Assume that $ 1< p \leq p^*$ if $3\leq N\leq4$ and $ 1< p< p^*$ if $N=2$. Then, there exist $T^*>0$ and a unique maximal solution to \eqref{S},
$$ {\bf u} \in C ([0, T^*), H).$$ Moreover,
\begin{enumerate}
\item ${\bf u}\in \big(L^{\frac{4p}{N(p-1)}}([0, T^*], W^{2,2p})\big)^{(m)};$
\item ${\bf u}$ satisfies conservation of the energy and the mass;
\item $T^*=\infty$ in the subcritical case $(1<p<p^*)$.
\end{enumerate}
\end{thm}
In the critical case, global existence for small data holds in the energy space.
\begin{thm}\label{glb}
Let $2<N\leq4$ and $p=p^*.$ There exists $\epsilon_0>0$ such that if $\Psi:=(\psi_1,...,\psi_m) \in H$ satisfies $ \displaystyle \sum_{j=1}^m\displaystyle\int_{\R^N}|\nabla \psi_j|^2\,dx\leq \epsilon_0$, the system \eqref{S} possesses a unique solution ${\bf u}\in C(\R, H)$.
\end{thm}
Second, the system \eqref{S} scatters in the energy space. Indeed, we show that every global solution of \eqref{S} is asymptotic, as $t\to\pm\infty,$ to a solution of the associated linear Schr\"odinger system.
\begin{thm}\label{t2}
Let $2\leq N\leq 4$ and $ p_*< p< p^*.$ Take ${\bf u}\in C(\R, H),$ a global solution to \eqref{S}. Then
$${\bf u}\in \big(L^{\frac{4p}{N(p - 1)}}(\R, W^{1, 2p})\big)^{(m)}$$
and there exists $\Psi:=(\psi_1,...,\psi_m)\in H$ such that
$$\lim_{t\longrightarrow\pm\infty}\|{\bf u}(t)-(e^{it\Delta}\psi_1,...,e^{it\Delta}\psi_m)\|_{H}=0.$$
\end{thm}
In the next subsection, we give some standard estimates needed in the paper.
\subsection{Tools}
We start with some properties of the free Schr\"odinger kernel.
\begin{prop}\label{fre}
Denoting the free operator associated to the fractional Schr\"odinger equation
$$e^{it\Delta}u_0:=\mathcal F^{-1}(e^{-it|y|^2})*u_0,$$
yields
\begin{enumerate}
\item
$e^{it\Delta}u_0$ is the solution to the linear problem associated to $(NLS)_p$;
\item
$e^{it\Delta}u_0 \pm i\int_0^te^{i(t-s)\Delta}u|u|^{p-1}\,ds$ is the solution to the problem $(NLS)_p$;
\item
$(e^{it\Delta})^*=e^{-it\Delta}$;
\item
$e^{it\Delta}$ is an isometry of $L^2$.
\end{enumerate}
\end{prop}
Now, we give the so-called Strichartz estimate \cite{Cas}.
\begin{defi}
A pair $(q,r)$ of positive real numbers is said to be admissible if
$$2\leq q,r\leq \infty,\quad (q,r,N) \neq(2, \infty, 2)\quad \mbox{and} \quad \frac{2}{q} = N\Big(\frac{1}{2} - \frac{1}{r}\Big).$$
\end{defi}
\begin{prop}Let two admissible pairs $(q,r),\, (a,b)$ and $T>0.$ Then, there exists a positive real number $C$ such that
\begin{equation}\label{S1}
\|u\|_{L_T^q(L^{r})}\leq C \Big( \|u_0\| + \|i\dot u+ \Delta u \|_{L_T^{ a^\prime}(L^{b^\prime})}\Big).
\end{equation}
\end{prop}
The following Morawetz estimate is essential in proving scattering \cite{cgt}.  
\begin{prop}\label{prop2''}
Let $2\leq N\leq4$, $1<p\leq p^*$ and ${\bf u}\in C(I,H)$ the solution to \eqref{S}. Then,
\begin{enumerate}
\item
if $N\geq4$,
\begin{equation}\label{mrwtz1}
\sum_{j=1}^m\int_I\int_{\R^N\times\R^N}\frac{|u_j(t,x)|^2|u_j(t,y)|^2}{|x-y|^3}dxdydt\lesssim_u1;
\end{equation}
\item
if $N=3$,
\begin{equation}\label{mrwtz2}
\sum_{j=1}^m\int_I\int_{\R^3}|u_j(t,x)|^4dxdt\lesssim_u1.\end{equation}
\item
if $N=2$,
\begin{equation}\label{mrwtz3}
\sum_{j=1}^m\int_I\|u_j(t)\|_{L^8(\R^2)}^4dt\lesssim_u1.\end{equation}
\end{enumerate}
\end{prop}
For the the reader convenience, a proof following as in \cite{cgt,cks}, is given in appendix. Let us gather some useful Sobolev embeddings \cite{AC1}.
\begin{prop}\label{injection}
The continuous injections hold
\begin{enumerate}
\item $ W^{s,p}(\R^N)\hookrightarrow L^q(\R^N)$ whenever
$1<p<q<\infty, \quad s>0\quad \mbox{and}\quad \frac{1}{p}\leq \frac{1}{q} + \frac {s}{N};$
\item $W^{s,p_1}(\R^N)\hookrightarrow W^{s - N(\frac{1}{p_1} - \frac{1}{p_2}),p_2}(\R^N)$ whenever $1\leq p_1\leq p_2 \leq \infty.$
\end{enumerate}
\end{prop}
We close this subsection with some absorption result \cite{Tao}.
\begin{lem} \label{Bootstrap}
 Let $T>0$ and $X\in C([0,  T], \R_+)$ such that
$$ X\leq a + b X^{\theta}\quad on \quad [0,T],$$
where  $a,\, b>0,\, \theta>1,\, a<(1 - \frac{1}{\theta})\frac{1}{(\theta b)^{\frac{1}{\theta}}}$ and $X(0)\leq \frac{1}{(\theta b)^{\frac{1}{\theta -1}}}.$ Then
$$X\leq \frac{\theta}{\theta - 1}a\quad on \quad [0, T].$$
\end{lem}
\section{Local well-posedness}
This section is devoted to prove Theorem \ref{existence}. The proof contains three steps. First we prove existence of a local solution to \eqref{S}, second we show uniqueness and finally we establish global existence in the subcritical case.
\subsection{Local existence}
We use a standard fixed point argument. For $T>0,$ we denote the space
$$E_T:=\big(C([0,T],H^1)\cap L^{\frac{4p}{N(p-1)}}([0, T], W^{1,2p})\big)^{(m)}$$ endowed with the complete norm
$$\|{\bf u}\|_T:=\displaystyle\sum_{j=1}^m\Big(\|u_j\|_{L_T^\infty(H^1)}+\| u_j\|_{L^{\frac{4p}{N(p-1)}}_T( W^{1,2p})}\Big).$$
Define the function
$$\phi({\bf u})(t) := T(t){\Psi} - i \displaystyle\sum_{k=1}^{m}\displaystyle\int_0^tT(t-s)\big(|u_k|^p|u_1|^{p-2}u_1,...,|u_k|^p|u_m|^{p-2}u_m\big)\,ds,$$
where $T(t){\Psi} := (e^{it\Delta}\psi_{1},...,e^{it\Delta}\psi_{m}).$ We prove the existence of some small $T, R >0$ such that $\phi$ is a contraction on the ball $ B_T(R)$ whith center zero and radius $R.$ Take ${\bf u}, {\bf v}\in E_T$ applying the Strichartz estimate \eqref{S1}, we get
$$\|\phi({\bf u}) - \phi({\bf v})\|_T\lesssim \displaystyle\sum_{j, k=1}^{m}\Big\||u_k|^p|u_j|^{p-2}u_j -  |v_k|^p |v_j|^{p-2}v_j\Big\|_{L^{\frac{4p}{p(4-N) + N}}(W^{1,{\frac{2p}{2p-1}}})}.$$
To derive the contraction, consider the function
$$ f_{j,k}: \C^m\rightarrow \C,\quad (u_1,...,u_m)\mapsto |u_k|^p|u_j|^{p-2}u_j.$$
With the mean value Theorem
$$|f_{j,k}({\bf u})-f_{j,k}({\bf v})|\lesssim\max\{ |u_k|^{p - 1}|u_j|^{p - 1}+{|u_k|^{p}|u_j|^{p-2}}, |v_k|^p|v_j|^{p - 2}+{|v_k|^{p - 1}|v_j|^{p - 1}}\}|{\bf u} - { \bf v}|.$$
Using H\"older inequality, Sobolev embedding and denoting the quantity
$$ (\mathcal{I}):=\| f_{j,k}({\bf u})-f_{j,k}({\bf v})\|_{L_T^{\frac{4p}{p(4-N) + N}}(L^{\frac{2p}{2p-1}})},$$ we compute via a symmetry argument
\begin{eqnarray*}
(\mathcal{I})
&\lesssim &\big\| \big(|u_k|^{p - 1}|u_j|^{p - 1} +|u_k|^p|u_j|^{p - 2}\big)|{\bf u} - { \bf v}|\big\|_{L_T^{\frac{4p}{p(4-N) + N}}(L^{\frac{2p}{2p-1}})} \\
&\lesssim&\|{\bf u} - { \bf v}\|_{L_T^{\frac{4p}{N(p-1)}}(L^{2p})} \big\| |u_k|^{p-1}|u_j|^{p -1} + |u_k|^{p}|u_j|^{p-2}  \big\|_{L_T^{\frac{4p}{4p - 2N(p-1)}}(L^{\frac{p}{p-1}})}\\
&\lesssim&T^{\frac{4p - 2N(p-1)}{4p}} \|{\bf u} - { \bf v}\|_{L_T^{\frac{4p}{N(p-1)}}(L^{2p})}\big\| |u_k|^{p-1}|u_j|^{p-1}
+ |u_k|^{p}|u_j|^{p-2}  \big\|_{L_T^\infty(L^{\frac{p}{p-1}})} \\
&\lesssim& T^{\frac{4p - 2N(p-1)}{4p}} \|{\bf u} - { \bf v}\|_{L_T^{\frac{4p}{N(p-1)}}(L^{2p})}\Big(\|u_k\|_{L_T^\infty(L^{2p})}^{p-1}\|u_j\|_{L_T^\infty(L^{2p})}^{p-1} + \|u_k\|_{L_T^\infty(L^{2p})}^p\|u_j\|_{L_T^\infty(L^{2p})}^{p-2} \Big)\\
&\lesssim& T^{\frac{4p - 2N(p-1)}{4p}} \|{\bf u} - { \bf v}\|_{L_T^{\frac{4p}{N(p-1)}}(L^{2p})}\Big(\|u_k\|_{L_T^\infty(H^1)}^{p-1}\|u_j\|_{L_T^\infty(H^1)}^{p-1}
+ \|u_k\|_{L_T^\infty(H^1)}^p\|u_j\|_{L_T^\infty(H^1)}^{p-2} \Big).
\end{eqnarray*}
Then
\begin{eqnarray*}
\displaystyle\sum_{k,j=1}^m\| f_{j,k}({\bf u})-f_{j,k}({\bf v})\|_{L_T^{\frac{4p}{p(4-N) + N}}(L^{\frac{2p}{2p-1}})}
&\lesssim & T^{\frac{4p - 2N(p-1)}{4p}} R^{2p-2}\|{\bf u} - {\bf v}\|_{T}.
\end{eqnarray*}
It remains to estimate the quantity
$$\big\|\nabla \big(f_{j,k}({\bf u}) - f_{j,k}({\bf v})\big)\big\|_{L_T^{\frac{4p}{p(4-N) + N}}(L^{\frac{2p}{2p-1}})}.$$
Write
\begin{eqnarray*}
\partial_i\Big((f_{j,k}({\bf u}) - f_{j,k}({\bf v})\Big)
&=& \Big(\partial_i{u}\partial_i (f_{j,k})({\bf u}) - \partial_i{v}\partial_i(f_{j,k})({\bf v})\Big)\\
& =&\partial_i({ u} - { v})\partial_i(f_{j,k})({\bf u}) +  \partial_i{ v}\Big(\partial_i(f_{j,k})({\bf u}) - \partial_i(f_{j,k})({\bf v})\Big).
\end{eqnarray*}
Thus
\begin{eqnarray*}
\big\|\nabla\Big(f_{j,k}({\bf u}) - f_{j,k}({\bf v})\Big)\big\|_{L_T^{\frac{4p}{p(4-N) + N}}(L^{\frac{2p}{2p-1}})}
&\leq&\big\| \displaystyle\sum_i\partial_i({ u} - { v})\partial_i(f_{j,k})({\bf u})  \big\|_{L_T^{\frac{4p}{p(4-N) + N}}(L^{\frac{2p}{2p-1}})}\\
& +& \big\|  \displaystyle\sum_i \partial_i{ v}\Big(\partial_i(f_{j,k})({\bf u}) - \partial_i(f_{j,k})({\bf v})\Big)\big\|_{L_T^{\frac{4p}{p(4-N) + N}}(L^{\frac{2p}{2p-1}})}\\
&\leq&(\mathcal{I}_1) + (\mathcal{I}_2).
\end{eqnarray*}
Thanks to H\"older inequality and Sobolev embedding, we obtain
\begin{eqnarray*}
(\mathcal{I}_1)
&\lesssim&\|\nabla({\bf u} - {\bf v})\|_{L_T^{\frac{4p}{N(p-1)}}(L^{2p})} \big\| |u_k|^{p-1}|u_j|^{p -1}+ {|u_k|^{p}|u_j|^{p-2}}\big\|_{L_T^{\frac{4p}{4p - 2N(p-1)}}(L^{\frac{p}{p-1}})}\\
&\lesssim&T^{\frac{4p - 2N(p-1)}{4p}} \|\nabla({\bf u} - {\bf v})\|_{L_T^{\frac{4p}{N(p-1)}}(L^{2p})}\big\| |u_k|^{p-1}|u_j|^{p-1} + |u_k|^{p}|u_j|^{p-2}  \big\|_{L_T^\infty(L^{\frac{p}{p-1}})} \\
&\lesssim& T^{\frac{4p - 2N(p-1)}{4p}} \|{\bf u} - {\bf v}\|_T\Big(\|u_k\|_{L_T^\infty(L^{2p})}^{p-1}\|u_j\|_{L_T^\infty(L^{2p})}^{p-1}
+ \|u_k\|_{L_T^\infty(L^{2p})}^p\|u_j\|_{L_T^\infty(L^{2p})}^{p-2} \Big)\\
&\lesssim& T^{\frac{4p - 2N(p-1)}{4p}} \|{\bf u} - {\bf v}\|_T\Big(\|u_k\|_{L_T^\infty(H^1)}^{p-1}\|u_j\|_{L_T^\infty(H^1)}^{p-1}
+ \|u_k\|_{L_T^\infty(H^1)}^p\|u_j\|_{L_T^\infty(H^1)}^{p-2} \Big).
\end{eqnarray*}
With the same way 
{\small \begin{eqnarray*}
(\mathcal{I}_2)
&\lesssim& \| \nabla {\bf v}\|_{L_T^{\frac{4p}{N(p-1)}}(L^{2p})} \| {\bf u} - {\bf v}\|_{L^\infty_T(L^{2p})}\big\||u_k|^{p-2}|u_j|^{p-1}+|u_k|^{p}|u_j|^{p-3}\big\|_{L_T^{\frac{4p}{4p - 2N(p-1)}}(L^{\frac{2p}{2p-3}})}\\
&\lesssim&T^{\frac{4p - 2N(p-1)}{4p}}\| \nabla {\bf v}\|_{L_T^{\frac{4p}{N(p-1)}}(L^{2p})} \| {\bf u} - {\bf v}\|_{L^\infty_T(L^{2p})}\big\||u_k|^{p-2}|u_j|^{p-1} + |u_k|^{p}|u_j|^{p-3}\big\|_{L_T^\infty(L^{\frac{2p}{2p-3}})}\\
&\lesssim&T^{\frac{4p - 2N(p-1)}{4p}}\| \nabla {\bf v}\|_{L_T^{\frac{4p}{N(p-1)}}(L^{2p})} \| {\bf u} - {\bf v}\|_{L^\infty_T(L^{2p})}\Big(\|u_k\|_{L_T^\infty(L^{2p})}^{p-2}\|u_j\|_{L_T^\infty(L^{2p})}^{p-1}
+ \|u_k\|_{L_T^\infty(L^{2p})}^p\|u_j\|_{L_T^\infty(L^{2p})}^{p-3} \Big)\\
&\lesssim&T^{\frac{4p - 2N(p-1)}{4p}}\| \nabla {\bf v}\|_{L_T^{\frac{4p}{N(p-1)}}(L^{2p})} \| {\bf u} - {\bf v}\|_{L^\infty_T({H^1})}\Big(\|u_k\|_{L_T^\infty({H^1})}^{p-2}\|u_j\|_{L_T^\infty({H^1})}^{p-1}
+ \|u_k\|_{L_T^\infty({H^1})}^p\|u_j\|_{L_T^\infty({H^1})}^{p-3} \Big).
\end{eqnarray*}}
Thus, for $T>0$ small enough, $\phi$ is a contraction satisfying
$$\|\phi({\bf u}) - \phi({\bf v})\|_T\lesssim T^{\frac{4p - N(p-1)}{4p}}R^{2p-3}\|{\bf u} - {\bf v}\|_T .$$
Taking in the last inequality ${\bf v}=0,$ yields
\begin{eqnarray*}
\|\phi({\bf u})\|_T
&\lesssim& T^{\frac{4p - N(p-1)}{4p}}R^{2p-2}+ \|\phi(0)\|_T\\
&\lesssim& T^{\frac{4p - N(p-1)}{4p}}R^{2p-2}+ TR .
\end{eqnarray*}
Since $p_*<p\leq p^*$ if $N\in\{3,4\}$ and $p_*<p< p^*$ if $N=2$, $\phi$ is a contraction of $ B_T(R)$ for some $R,T>0$ small enough.
\subsection{Uniqueness}
In what follows, we prove uniqueness of solution to the Cauchy problem \eqref{S}. Let $T>0$ be a positive time, ${\bf u},{\bf v}\in C_T(H)$ two solutions to \eqref{S} and ${\bf w} := {\bf u} - {\bf v}.$ Then
$$i\dot w_j +\Delta w_j = \displaystyle \sum_{k=1}^{m}\big( |u_k|^p|u_j|^{p - 2 }u_j -  |v_k|^p|v_j|^{p - 2 }v_j\big),\quad w_j(0,.)= 0.$$
Applying Strichartz estimate with the admissible pair $(q,r) = (\frac{4p}{N(p-1)}, 2p) $ and denoting for simplicity $L_T^q(L^r)$ the norm of $(L_T^q(L^r))^{(m)}$, we have
\begin{eqnarray*}
\|{\bf u} - {\bf v}\|_{L_T^q(L^r)}\lesssim \displaystyle\sum_{j,k=1}^{m}\big\|f_{j,k}({\bf u}) -  f_{j,k}({\bf v})\big\|_{L_T^{q^\prime}(L^{r^\prime})}.
\end{eqnarray*}
Taking $T>0$ small enough, whith a continuity argument, we may assume that
$$ \max_{j=1,...,m}\|u_j\|_{L_T^\infty(H^1)}\leq 1.$$
Using previous computation with$$ (\mathcal{I}) :=\big\|f_{j,k}({\bf u}) -  f_{j,k}({\bf v})\big\|_{L_T^{q^\prime}(L^{r^\prime})}=  \big\||u_k|^p|u_j|^{p-2}u_j - |v_k|^p|v_j|^{p-2}v_j\big\|_{L_T^{q^\prime}(L^{r^\prime})},$$
we have
\begin{eqnarray*}
(\mathcal{I})&\lesssim&\big\|\Big(|u_k|^{p-1}|u_j|^{p-1}  + |u_k|^p|u_j|^{p-2} \Big)|{\bf u} - {\bf v}|\big\|_{L_T^{\frac{4p}{p(4-N) + N}}(L^{\frac{2p}{2p-1}})}\\
&\lesssim&\|{\bf u} - {\bf v}\|_{L_T^{\frac{4p}{p(4-N) + N}}(L^{2p})}\big\| |u_k|^{p-1}|u_j|^{p-1} + |u_k|^p|u_j|^{p-2} \big\|_{L_T^\infty(L^{\frac{p}{p-1}})}\\
&\lesssim& T^{\frac{(4 - N)p + N}{4 p}}\|{\bf u} - {\bf v}\|_{L_T^{\frac{4p}{N(p - 1)}}(L^{2p})}\Big(\|u_k\|_{L_T^\infty(H^1)}^{p-1} \|u_j\|_{L_T^\infty(H^1)}^{p-1}+ \|u_k\|_{L_T^\infty(H^1)}^{p}\|u_j\|_{L_T^\infty(H^1)}^{p-2}   \Big).
\end{eqnarray*}
Then
$$ \|{\bf w}\|_{L_T^q(L^r)}\lesssim  T^{\frac{(4 - N)p + N}{4 p}}\|{\bf w} \|_{L_T^q(L^r)}.$$
Uniqueness follows for small time and then for all time with a translation argument.
\subsection{Global existence in the subcritical case}
 The global existence is a consequence of energy conservation and previous calculations. Let ${\bf u} \in C([0, T^*), H)$ be the unique maximal solution of \eqref{S}. We prove that ${\bf u}$ is global. By contradiction, suppose that $T^*<\infty.$ Consider for $0< s <T^*,$ the problem
$$(\mathcal{P}_s)
\left\{
\begin{array}{ll}
i\dot v_j +\Delta v_j = \displaystyle \sum_{k,j=1}^{m} |v_k|^p|v_j|^{p - 2 }v_j;\\
v_j(s,.) = u_j(s,.).
\end{array}
\right.
$$
Using the same arguments used in the local existence, we can prove a real $\tau>0$ and a solution ${\bf v} = (v_1,...,v_m)$ to $(\mathcal{P}_s)$ on $C\big([s, s+\tau], H).$ Using the conservation of energy we see that $\tau$ does not depend on $s.$ Thus, if we let $s$ be close to $T^*$ such that $T^*< s + \tau,$ this fact contradicts the maximality of $T^*.$
\section{Global existence in the critical case}
In this section $N\in\{3,4\}$. We establish global existence of a solution to \eqref{S} in the critical case $p=p^*$ for small data as claimed in Theorem \ref{glb}.\\
Several norms have to be considered in the analysis of the critical case. Letting $I\subset \R$ a time slab, we define 
\begin{eqnarray*}
\|u\|_{M(I)} &:= &\|\nabla u\|_{L^{\frac{2(N + 2)}{N-2}}(I, L^{\frac{2N(N + 2)}{N^2  + 4}})};\\
\|u\|_{S(I)}& :=&\| u\|_{L^{\frac{2(N + 2)}{N-2}}(I, L^{\frac{2(N + 2)}{N -2}})}.
\end{eqnarray*}
Let $M(\R)$ be the completion of $C_c^\infty(\R^{N+1})$ endowed with the norm $\|.\|_{M(\R)},$ and $M(I)$ be the set consisting of the restrictions to $I$ of functions in $M(\R).$ An important quantity closely related to the mass and the energy, is the functional $\xi$ defined for ${\bf u}\in H $ by
$$\xi({\bf u}) = \displaystyle \sum_{j=1}^m\displaystyle\int_{\R^N}|\nabla u_j|^2\,dx.$$
We give an auxiliary result.
\begin{prop}\label{proposition 1}
Let $p= p^*$, $\Psi:=(\psi_1,..,\psi_m)\in\dot H:=(\dot H^1)^{(m)}$ and $A:=\|\Psi\|_{\dot H}$. There exists $\delta:=(\delta_A)>0$ such that for any interval $I=[0, T],$ if
$$ \displaystyle\sum_{j=1}^{m}\|e^{it\Delta}\psi_{j}\|_{S(I)}< \delta,$$
then there exits a unique solution ${\bf u}\in C(I, H)$ of \eqref{S} which satisfies ${\bf u}\in \big(M(I)\cap L^{\frac{2(N+2)}{N}}(I\times \R^N)\big)^{(m)}.$ Moreover,
\begin{gather*}
\displaystyle\sum_{j=1}^{m}\|u_j\|_{S(I)}\leq 2\delta.
\end{gather*}
Besides, the solution depends continuously on the initial data in the sense that there exists $\delta_0$ depending on $\delta,$ such that for any $\delta_1\in (0,\delta_0),$ if $\displaystyle\sum_{j=1}^{m}\|\psi_{j} - \varphi_{j}\|_{H^1}\leq \delta_1$ and ${\bf v}$ is the local solution of \eqref{S} with initial data $\varphi:=(\varphi_{1},...,\varphi_{m}),$ then ${\bf v}$ is defined on $I$ and for any admissible couple $(q,r)$,
$$\|{\bf u} - {\bf v}\|_{(L^q(I, L^r)\cap \dot H)^{(m)}}\leq C\delta_1.$$
\end{prop}
\begin{proof}
The proposition follows from a contraction mapping argument. For ${\bf u}\in( W(I))^{(m)}$, we let the function
$$\phi({\bf u})(t) := T(t){\Psi} -i \displaystyle\sum_{k =1}^{m}\displaystyle\int_0^tT(t-s)\Big(|u_k|^{\frac{N}{N-2}}|u_1|^{\frac{4-N}{N-2}}u_1,..,|u_k|^{\frac{N}{N-2}}|u_m|^{\frac{4-N}{N-2}}u_m\Big)\,ds.$$
Define $A:=\|\Psi\|_{\dot H}$ and the set
$$ X_{a,b} := \Big\{ {\bf u}\in (M(I))^{(m)};\, \displaystyle\sum_{j=1}^{m}\|u_j\|_{M(I)}\leq a, \, \displaystyle\sum_{j=1}^{m}\|u_j\|_{S(I)}\leq b\Big\}$$
where $a,b>0$ are sufficiently small to fix later. Using Strichartz estimate, we get
\begin{eqnarray*}
\|\phi({\bf u}) - \phi({\bf v})\|_{M(I)}
&\lesssim&\sum_{j,k=1}^{m}\big\|\nabla(f_{j,k}({\bf u}) - f_{j,k}({\bf v}))\big\|_{L^{\frac{2(N+2)}{N+4}}(I,L^{\frac{2(N+2)}{N+4}})}.
\end{eqnarray*}
Write
\begin{eqnarray*}
\partial_i\Big((f_{j,k}({\bf u}) - f_{j,k}({\bf v})\Big)
&=& \Big(\partial_i{u} \partial_i(f_{j,k})({\bf u}) -\partial_i {v}\partial_i(f_{j,k})({\bf v})\Big)\\
& =&\partial_i({ u} - { v})\partial_i(f_{j,k})({\bf u}) +\partial_i { v}\Big(\partial_i(f_{j,k})({\bf u}) -\partial_i (f_{j,k})({\bf v})\Big).
\end{eqnarray*}
Thus
\begin{eqnarray*}
\big\|\nabla\Big(f_{j,k}({\bf u}) - f_{j,k}({\bf v})\Big)\big\|_{L^{\frac{2(N+2)}{N+4}}(I,L^{\frac{2(N+2)}{N+4}})}
&\leq&\big\| \displaystyle\sum_i\partial_i({ u} - { v})\partial_i(f_{j,k})({\bf u})  \big\|_{L^{\frac{2(N+2)}{N+4}}(I,L^{\frac{2(N+2)}{N+4}})}\\
& +& \big\|\displaystyle\sum_i\partial_i { v}\Big(\partial_i(f_{j,k})({\bf u}) - \partial_i(f_{j,k})({\bf v})\Big)\big\|_{L^{\frac{2(N+2)}{N+4}}(I,L^{\frac{2(N+2)}{N+4}})}\\
&\leq&(\mathcal{I}_1) + (\mathcal{I}_2).
\end{eqnarray*}
Using H\"older inequality and Sobolev embedding, yields
\begin{eqnarray*}
(\mathcal{I}_1)
&\lesssim& \big\| |\nabla({\bf u}-{\bf v})| \Big( |u_k|^{\frac{2}{N-2}}|u_j|^{\frac{2}{N-2}} + |u_k|^{\frac{N}{N-2}}|u_j|^{\frac{4-N}{N-2}}\Big)\big\|_{L_T^{2}(L^{\frac{2N}{N+2}})}\\
&\lesssim&\|\nabla({\bf u}-{\bf v})\|_{L_T^{\frac{2(N+2)}{N -2}}(L^{\frac{2N(N+2)}{N^2+4}})}\big\||u_k|^{\frac{2}{N-2}}|u_j|^{\frac{2}{N-2}} + |u_k|^{\frac{N}{N-2}}|u_j|^{\frac{4-N}{N-2}}\big\|_{L_T^{\frac{N+2}2}(L^{\frac{N+2}2})}\\
&\lesssim&\|\nabla({\bf u}-{\bf v})\|_{L_T^{\frac{2(N+2)}{N - 2}}(L^{\frac{2N(N+2)}{N^2+4}})}\Big(\|u_k\|_{L_T^{\frac{2(N+2)}{N- 2}}(L^{\frac{2(N+2)}{N - 2}})}^{\frac{2}{N-2}} \|u_j\|_{L_T^{\frac{2(N+2)}{N - 2}}(L^{\frac{2(N+2)}{N- 2}})}^{\frac{2}{N-2}}\\ &+& \|u_k\| _{L_T^{\frac{2(N+2)}{N - 2}}(L^{\frac{2(N+2)}{N - 2}})} ^{\frac{N}{N-2}}\|u_j\|_{L_T^{\frac{2(N+2)}{N - 2}}(L^{\frac{2(N+2)}{N - 2}})}^{\frac{4-N}{N-2}}\Big)\\
&\lesssim&\|{\bf u}-{\bf v}\|_{(M(I))^{(m)}}\Big(\|u_k\|_{S(I)}^{\frac{2}{N-2}} \|u_j\|_{S(I)}^{\frac{2}{N-2}}+\|u_k\|_{S(I)} ^{\frac{N}{N-2}}\|u_j\|_{S(I)}^{\frac{4-N}{N-2}}\Big)\\
&\lesssim&\|{\bf u}-{\bf v}\|_{(M(I))^{(m)}}\|{\bf u}\|_{(S(I))^m}^{\frac{4}{N-2}}.
\end{eqnarray*}
Using H\"older inequality and Sobolev embedding, yields
\begin{eqnarray*}
(\mathcal{I}_2)
&\lesssim& \big\| |\nabla{\bf u}|({\bf u}-{\bf v})| \Big( |u_k|^{\frac{4-N}{N-2}}|u_j|^{\frac{2}{N-2}} + |u_k|^{\frac{N}{N-2}}|u_j|^{\frac{6-2N}{N-2}}\Big)\big\|_{L_T^{2}(L^{\frac{2N}{N+2}})}\\
&\lesssim&\|\nabla{\bf u}\|_{L_T^{\frac{2(N+2)}{N -2}}(L^{\frac{2N(N+2)}{N^2+4}})}\|{\bf u}-{\bf v}\|_{S(I)^m}\Big(\|u_k\|_{L_T^{\frac{2(N+2)}{N- 2}}(L^{\frac{2(N+2)}{N - 2}})}^{\frac{4-N}{N-2}} \|u_j\|_{L_T^{\frac{2(N+2)}{N - 2}}(L^{\frac{2(N+2)}{N- 2}})}^{\frac{2}{N-2}}\\ &+& \|u_k\| _{L_T^{\frac{2(N+2)}{N - 2}}(L^{\frac{2(N+2)}{N - 2}})} ^{\frac{N}{N-2}}\|u_j\|_{L_T^{\frac{2(N+2)}{N - 2}}(L^{\frac{2(N+2)}{N - 2}})}^{\frac{6-2N}{N-2}}\Big)\\
&\lesssim&\|{\bf u}\|_{(M(I))^{(m)}}\|{\bf u}-{\bf v}\|_{(S(I))^m}\|{\bf u}\|_{(S(I))^m}^{\frac{6-N}{N-2}}.
\end{eqnarray*}
Then
\begin{eqnarray*}
\|\phi({\bf u}) - \phi({\bf v})\|_{(M(I))^{(m)}}
&\lesssim& a^{\frac{4}{N-2}}\|{\bf u-v}\|_{(M(I))^m}+ba^{\frac{6-N}{N-2}}\|{\bf u-v}\|_{(S(I))^m}\\
&\lesssim& (a^{\frac{4}{N-2}}+ba^{\frac{6-N}{N-2}})\|{\bf u-v}\|_{(M(I))^m}.
\end{eqnarray*}
Moreover, taking in the previous inequality ${\bf v=0}$, we get for small $\delta>0$,
\begin{gather*}
\|\phi({\bf u})\|_{(S(I))^{(m)}}\leq\delta+Ca^{\frac{4}{N-2}};\\
\|\phi({\bf u})\|_{(M(I))^{(m)}}\leq CA+Cba^{\frac{4}{N-2}}.
\end{gather*}
With a classical Picard argument, for small $a=2\delta,b>0$, there exists ${\bf u}\in X_{a,b}$ a solution to \eqref{S} satisfying
 $$\|{\bf u}\|_{(S(I))^{(m)}}\leq 2\delta.$$
The rest of the Proposition is a consequence of the fixed point properties.
\end{proof}
We are ready to prove Theorem \ref{glb}.
\begin{proof}[{\bf Proof of Theorem \ref{glb}}]Using the previous proposition via the fact that
$$\|e^{it\Delta}\Psi\|_{S(I)}\lesssim\|e^{it\Delta}\Psi\|_{M(I)}\lesssim\|\Psi\|_{\dot H},$$
it suffices to prove that $\|{\bf u}\|_{\dot H}$ remains small on the whole interval of existence of ${\bf u}.$ 
Write with conservation of the energy and Sobolev's inequality
\begin{eqnarray*}
\|{\bf u}\|_{\dot H}^2&\leq& 2E(\Psi) +\frac{N -2}{N}\displaystyle \sum_{j,k=1}^{m}\displaystyle \int_{\R^N} |u_j(x,t)|^{\frac{N}{N - 2}} |u_k(x,t)|^{\frac{N}{N - 2}}\,dx \\
&\leq& C\big(  \xi(\Psi) + \xi(\Psi)^{\frac{N}{N - 2}}\big) + C \big(\displaystyle\sum_{j=1}^{m}\|\nabla u_j\|_2^2\big)^{\frac{N}{N - 2}}\\
&\leq& C\big(  \xi(\Psi) + \xi(\Psi)^{\frac{N}{N - 2}}\big) +C\|{\bf u}\|_{\dot H} ^{\frac{2N}{N - 2}}.
\end{eqnarray*}
So by Lemma \ref{Bootstrap}, if $\xi(\Psi)$ is sufficiently small, then ${\bf u}$ stays small in the ${\dot H}$ norm.
\end{proof}
\section{Scattering}
In this section, for any time slab $I,$ take the Strichartz space
$$S(I):=C(I, H^1)\cap{L^{\frac{4p}{N(p -1)}}(I, W^{1, 2p})}$$
endowed with the complete norm
$$ \|u\|_{S(I)}:= \|u\|_{L^\infty(I, H^1)} + \|u\|_{L^{\frac{4p}{N(p -1)}}(I, W^{1, 2p})}.$$
The first intermediate result is the following.
\begin{lem}
For any time slab $I,$ we have
$$\| {\bf u}(t) - e^{it\Delta}\Psi\|_{(S(I))^{(m)}}\lesssim\|{\bf u}\|_{L^\infty(I, L^{2p})}^{2p-\frac{4p}{N(p-1)}}\|{\bf u}\|_{L^{\frac{4p}{N(p-1)}}(I, W^{1,2p})}^{\frac{4p}{N(p-1)}-1}.$$
\end{lem}
\begin{proof}
Using Strichartz estimate, we have
$$\| {\bf u}(t) - e^{it\Delta}\Psi\|_{(S(I))^{(m)}}\lesssim \displaystyle\sum_{j,k=1}^m \|f_{j,k}({\bf u})\|_{L^{\frac{4p}{p(4 - N) + N}}(I, W^{1,\frac{2p}{2p -1}})}.$$
Thanks to H\"older inequality, we get
\begin{eqnarray*}
\|f_{j,k}({\bf u})\|_{L^\frac{2p}{2p -1}}&\lesssim&\big\||u_k|^p|u_j|^{p - 1}\big\|_{L^\frac{2p}{2p -1}}
\lesssim\|u_k\|_{L^{2p}}^p\|u_j\|_{L^{2p}}^{p -1}.
\end{eqnarray*}
Letting $\mu := \frac{4p - N(p - 1)}{2N(p - 1) }$, we get $p - \mu ={\frac{N(p -1)(2p +1) - 4p}{2N(p -1)}}$ and $ p - 1 -\mu={\frac{N(p -1)(2p -1) - 4p}{2N(p -1)}}$. Moreover,
\begin{eqnarray*}
\|f_{j,k}({\bf u})\|_{L^{\frac{4p}{p(4 - N) + N}}(I, L^{\frac{2p}{2p -1}})}&\lesssim& \big\|\|u_k\|_{L^{2p}}^p\|u_j\|_{L^{2p}}^{p -1} \big\|_{L^{\frac{4p}{p(4 - N) + N}}(I)}\\
&\lesssim&\|u_k\|_{L^\infty(I,L^{2p})}^{p -\mu}\|u_j\|_{L^\infty(I,L^{2p})}^{p -1-\mu}\big\|\|u_k\|_{L^{2p}}^\mu\|u_j\|_{L^{2p}}^{\mu} \big\|_{L^{\frac{4p}{p(4 - N) + N}}(I)}\\
&\lesssim&\|u_k\|_{L^\infty(I, L^{2p})}^{p -\mu}\|u_j\|_{L^\infty(I, L^{2p})}^{p -1-\mu}\|u_k\|_{L^{\frac{4p}{N(p -1)}}(I, L^{2p})}^{\mu}\|u_j\|_{L^{\frac{4p}{N(p -1)}}(I, L^{2p})}^{\mu}.
\end{eqnarray*}
Then,
\begin{eqnarray}
\displaystyle\sum_{j,k=1}^m\|f_{j,k}({\bf u})\|_{L^{\frac{4p}{p(4 - N) + N}}(I, L^{\frac{2p}{2p -1}})}&\lesssim&\displaystyle\sum_{j,k=1}^m\|u_k\|_{L^\infty(I, L^{2p})}^{p -\mu}\|u_j\|_{L^\infty(I, L^{2p})}^{p -1-\mu}\|u_k\|_{L^{\frac{4p}{N(p -1)}}(I, L^{2p})}^{\mu}\|u_j\|_{L^{\frac{4p}{N(p -1)}}(I, L^{2p})}^{\mu}\nonumber\\
&\lesssim&\displaystyle\sum_{k=1}^m\|u_k\|_{L^\infty(I, L^{2p})}^{p -\mu}\|u_k\|_{L^{\frac{4p}{N(p -1)}}(I, L^{2p})}^{\mu}\displaystyle\sum_{j=1}^m\|u_j\|_{L^\infty(I, L^{2p})}^{p -1 - \mu}\|u_j\|_{L^{\frac{4p}{N(p -1)}}(I, L^{2p})}^{\mu}\nonumber\\
&\lesssim&\Big(\displaystyle\sum_{k=1}^m\big(\|u_k\|_{L^\infty(I, L^{2p})}^{p -\mu}\big)^2\Big)^{\frac{1}{2}}\Big(\displaystyle\sum_{k=1}^m\big( \|u_k\|_{L^{\frac{4p}{N(p -1)}}(I, L^{2p})}^{\mu}\big)^2\Big)^{\frac{1}{2}}\nonumber\\
&\times&\Big(\displaystyle\sum_{j=1}^m\big(\|u_j\|_{L^\infty(I, L^{2p})}^{p - 1- \mu}\big)^2\Big)^{\frac{1}{2}}\Big(\displaystyle\sum_{j=1}^m\big( \|u_j\|_{L^{\frac{4p}{N(p -1)}}(I, L^{2p})}^{\mu}\big)^2\Big)^{\frac{1}{2}}\nonumber\\
&\lesssim&\|{\bf u}\|_{L^\infty(I, L^{2p})}^{2p-\frac{4p}{N(p - 1)}}\|{\bf u}\|_{L^{\frac{4p}{N(p -1)}}(I, L^{2p})}^{\frac{4p - N(p - 1)}{N(p -1)}}\label{sct1}.
\end{eqnarray}
It remains to estimate the quantity $(\mathcal{I}):=\|\nabla (f_{j,k}({\bf u}))\|_{{L^{\frac{4p}{p(4 - N) + N}}(I, L^{\frac{2p}{2p -1}})}}.$ 
Using H\"older inequality, we obtain
\begin{eqnarray*}
\|\nabla{\bf u} (f_{j,k})_i({\bf u})\|_{L^{\frac{2p}{2p -1}}}
&\lesssim&\big\| \nabla{\bf u}\big(|u_k|^{p-1}|u_j|^{p-1} + |u_k|^p|u_j|^{p-2}\big)\big\|_{L^{\frac{2p}{2p -1}}}\\
&\lesssim&\|\nabla{\bf u}\|_{(L^{2p})^m}\Big(\big\||u_k|^{p-1}|u_j|^{p - 1}\big\|_{L^\frac{p}{p - 1}}   + \big\||u_k|^{p }|u_j|^{p - 2}\big\|_{L^\frac{p}{p - 1}} \Big)\\
&\lesssim&\|\nabla{\bf u}\|_{(L^{2p})^m}\Big(\|u_k\|_{L^{2p}}^{p-1} \|u_j\|_{L^{2p}}^{p-1} + \|u_k\|_{L^{2p}}^{p}\|u_j\|_{L^{2p}}^{p-2}  \Big).
\end{eqnarray*}
Letting $\theta = :\frac{2p - N(p -1)}{N(p - 1)},$ we get
\begin{eqnarray*}
(\mathcal{I})
 &\lesssim& \Big\|\|\nabla{\bf u}\|_{(L^{2p})^m}\Big(\|u_k\|_{L^{2p}}^{p-1} \|u_j\|_{L^{2p}}^{p-1} + \|u_k\|_{L^{2p}}^{p}\|u_j\|_{L^{2p}}^{p-2}  \Big)\Big\|_{L^{\frac{4p}{p(4 - N) +N}}}\\
&\lesssim&\|\nabla {\bf u}\|_{L^{\frac{4p}{N(p -1)}}(I, L^{2p})}\Big(\big\|\|u_k\|_{L^{2p}}^{p-1} \|u_j\|_{L^{2p}}^{p-1}\big\|_{L^{\frac{4p}{4p - 2N(p - 1)}}} + \big\|\|u_k\|_{L^{2p}}^{p}\|u_j\|_{L^{2p}}^{p-2}\big\|_{L^{\frac{4p}{4p - 2N(p - 1)}}}\Big)\\
&\lesssim&\|\nabla {\bf u}\|_{L^{\frac{4p}{N(p -1)}}(I, L^{2p})}\Big( \|u_k\|_{L^\infty(I, L^{2p})}^{p -1 -\theta}\|u_j\|_{L^\infty(I, L^{2p})}^{p -1- \theta}\big\|\|u_k\|_{L^{2p}}^{\theta} \|u_j\|_{L^{2p}}^{\theta}\big\|_{L^{\frac{4p}{4p - 2N(p - 1)}}} \\
&+&\|u_k\|_{L^\infty(I, L^{2p})}^{p - \theta}\|u_j\|_{L^\infty(I, L^{2p})}^{p -2-\theta}\big\|\|u_k\|_{L^{2p}}^{\theta} \|u_j\|_{L^{2p}}^{\theta}\big\|_{L^{\frac{4p}{4p - 2N(p - 1)}}} \Big)\\
&\lesssim& \|\nabla {\bf u}\|_{L^{\frac{4p}{N(p -1)}}(I, L^{2p})}\Big( \|u_k\|_{L^\infty(I, L^{2p})}^{p -1 -\theta}\|u_j\|_{L^\infty(I, L^{2p})}^{p -1- \theta}\|u_k\|_{L^{\frac{4p}{N(p -1)}}(I, L^{2p})}^{\theta} \|u_j\|_{L^{\frac{4p}{N(p -1)}}(I, L^{2p})}^{\theta} \\
&+&\|u_k\|_{L^\infty(I, L^{2p})}^{p - \theta}\|u_j\|_{L^\infty(I, L^{2p})}^{p -2-\theta}\|u_k\|_{L^{\frac{4p}{N(p -1)}}(I, L^{2p})}^{\theta} \|u_j\|_{L^{\frac{4p}{N(p -1)}}(I, L^{2p})}^{\theta} \Big)\\
&\lesssim& \|\nabla {\bf u}\|_{L^{\frac{4p}{N(p -1)}}(I, L^{2p})}\|{\bf u}\|_{L^\infty(I, L^{2p})}^{2(p -1 -\theta)}\|{\bf u}\|_{L^{\frac{4p}{N(p -1)}}(I, L^{2p})}^{2\theta}\\
&\lesssim& \|{\bf u}\|_{L^{\frac{4p}{N(p -1)}}(I, W^{1,2p})}^{\frac{4p}{N(p-1)}-1}\|{\bf u}\|_{L^\infty(I, L^{2p})}^{2p-\frac{4p}{N(p-1)}}\label{sct3}.
\end{eqnarray*}
\end{proof}{}
The next auxiliary result is about the decay of solution.
\begin{lem}\label{t1}
For any $2<r<\frac{2N}{N - 2},$ we have
$$\displaystyle\lim_{t\to \infty}\|{\bf u}(t)\|_{L^r}= 0.$$
\end{lem}
\begin{proof}
Let $\chi \in C^\infty_0(\R^N)$ be a cut-off function and $\varphi_n:=(\varphi_1^n,...,\varphi_m^n)$ be a sequence in $H$ satisfying $\displaystyle\sup_{n}\|\varphi_n\|_{H}<\infty$ and
$$ \varphi_n\rightharpoonup \varphi := (\varphi_1,...,\varphi_m)\in H.$$
Let ${\bf u}_n:=(u_1^n,...,u_m^n)\; (\mbox{respectively}\; {\bf u}:=(u_1,...,u_m))$ be the solution in $C(\R, H)$ to \eqref{S} with initial data $\varphi_n\, (\mbox{respectively}\; \varphi).$ In what follows, we prove a claim.\\
{\bf Claim.}\\
 For every $\epsilon>0,$ there exist $T_\epsilon>0$ and $ n_\epsilon\in \N$ such that
 \begin{equation}\label{chi} \|\chi({\bf u}_n - {\bf u})\|_{L_{T_\epsilon}^\infty  (L^2)}<\epsilon \quad \mbox{for all}\; n>n_\epsilon.\end{equation}
Indeed, introducing the functions ${\bf v}_n=(v_1^n,..,v_m^n):= \chi {\bf u}_n$ and ${\bf v}=(v_1,..,v_m) :=\chi {\bf u}.$ We compute, $v_j^n(0,.)  = \chi \varphi_j^n$ and
$$i\dot v_j^n + \Delta v_j^n
 = \Delta\chi u_j^n + 2 \nabla\chi \nabla u_j^n + \chi\big(\displaystyle\sum_{k=1}^m|u_k^n|^p|u_j^n|^{p - 2}u_j^n\big).$$
Similarly, $v_j(0,.)=\chi\phi_j$ and
$$ i\dot v_j + \Delta v_j = \Delta\chi u_j + 2 \nabla \chi \nabla u_j +  \chi\big(\displaystyle\sum_{k=1}^m|u_k|^p|u_j|^{p - 2}u_j\big).$$
Denoting ${\bf w}_n:= {\bf v}_n - {\bf v}$ and ${\bf z}_n:= {\bf u}_n - {\bf u},$ we have
$$ i\dot w_j^n + \Delta w_j^n = \Delta\chi z_j^n+ 2 \nabla \chi \nabla z_j^n + \chi\big(\displaystyle\sum_{k=1}^m|u_k^n|^p|u_j^n|^{p - 2}u_j^n - \displaystyle\sum_{k=1}^m|u_k|^p|u_j|^{p - 2}u_j\big).$$
By Strichartz estimate, we obtain
\begin{eqnarray*}
\|{\bf w}_n\|_{L_T^\infty(L^2) \cap L^{\frac{4p}{N(p-1)}}_T(L^{2p})}&\lesssim& \|\chi(\varphi_n - \varphi)\|_{L^2} + \|\Delta\chi {\bf z}_n\|_{L^1_T(L^2)}+ 2 \|\nabla \chi \nabla {\bf z}_n\|_{L^1_T(L^2)}\\
&+&\displaystyle\sum_{j,k=1}^m\big\|\chi\big(|u_k^n|^p|u_j^n|^{p - 2}u_j^n - |u_k|^p|u_j|^{p - 2}u_j\big)\big\|_{L^{\frac{4p}{p(4-N) + N}}_T(L^{\frac{2p}{2p-1}})}.
\end{eqnarray*}
Thanks to the Rellich Theorem, up to subsequence extraction, we have
$$\epsilon:=\|\chi(\varphi_n - \varphi)\|_{L^2}\longrightarrow0\quad\mbox{as}\quad n\longrightarrow\infty.$$
Moreover, by the conservation laws via properties of $\chi$,
\begin{eqnarray*}
 \mathcal{I}_1
&:=&\|\Delta\chi {\bf z}_n\|_{L^1_T(L^2)}+ 2 \|\nabla\chi \nabla {\bf z}_n\|_{L^1_T(L^2)}\\
&\lesssim&\| {\bf z}_n\|_{L^1_T(L^2)}+  \| \nabla {\bf z}_n\|_{L^1_T(L^2)}\\
&\lesssim& CT,
\end{eqnarray*}
where $$ C:= \|{\bf u}\|_{L^\infty(\R,H)} + \|{\bf u}_n\|_{L^\infty(\R,H)} .$$
Arguing as previously, we have
\begin{eqnarray*}
\mathcal{I}_2&:=&\|\chi(|u_k^n|^p|u_j^n|^{p-2}u_j^n - |u_k|^p|u_j|^{p-2}u_j)\|_{L^{\frac{4p}{p(4 - N) + N}}_T(L^{\frac{2p}{2p -1}})}\\
&\lesssim&\|\chi(|u_k^n|^{p -1}|u_j^n|^{p-1} - |u_k|^p|u_j|^{p-2})|{u}_j^n - { u}_j|\|_{L^{\frac{4p}{p(4 - N) + N}}_T(L^{\frac{2p}{2p -1}})}\\
&\lesssim&\|\chi({\bf u}_n  - {\bf u})\|_{L^{\frac{4p}{p(4 - N) + N}}_T(L^{2p})}\Big( \|u_k^n\|_{L^\infty_T(L^{2p})} ^{p-1}\|u_j^n\|_{L^\infty_T(L^{2p})} ^{p-1}  + \|u_k\|_{L^\infty_T(L^{2p})} ^{p}\|u_j\|_{L^\infty_T(L^{2p})} ^{p-2}\Big)\\
&\lesssim&T^{\frac{4p - 2N(p-1)}{4p}}\|{\bf w}_n \|_{L^{\frac{4p}{N(p -1)}}_T(L^{2p})}\Big( \|u_k^n\|_{L^\infty_T(H^1)} ^{p-1}\|u_j^n\|_{L^\infty_T(H^1)} ^{p-1}  + \|u_k\|_{L^\infty_T(H^1)} ^{p}\|u_j\|_{L^\infty_T(H^1)} ^{p-2}\Big)\\
&\lesssim&T^{\frac{4p - 2N(p-1)}{4p}}\|{\bf w}_n \|_{L^{\frac{4p}{N(p -1)}}_T(L^{2p})}\Big( \|u_k^n\|_{L^\infty_T(H^1)} ^{2(p-1)} + \|u_j^n\|_{L^\infty_T(H^1)} ^{2(p-1)}  + \|u_k\|_{L^\infty_T(H^1)} ^{2p} + \|u_j\|_{L^\infty_T(H^1)} ^{2(p-2)}\Big)\\
&\lesssim&T^{\frac{4p - 2N(p-1)}{4p}}\|{\bf w}_n \|_{L^{\frac{4p}{N(p -1)}}_T(L^{2p})}.
\end{eqnarray*}
As a consequence
\begin{eqnarray*}
\|{\bf w}_n\|_{L_T^\infty(L^2) \cap L^{\frac{4p}{N(p-1)}}_T(L^{2p})}
&\lesssim& \epsilon + CT + T^{\frac{4p - 2N(p-1)}{4p}}\|{\bf w}_n \|_{L^{\frac{4p}{N(p -1)}}(L^{2p})}\\
&\lesssim& \frac{\epsilon + T}{1 - T^{\frac{4p - 2N(p-1)}{4p}}}.
\end{eqnarray*}
The claim is proved.\\
By an interpolation argument it is sufficient to prove the decay for $r:= 2 +\frac{4}{N}.$ We recall the following Gagliardo-Nirenberg inequality
\begin{equation}\label{GN}
\|u_j(t)\|_{L^{2 + \frac{4}{N}}}^{2 + \frac{4}{N}}\leq C \|u_j(t)\|_{H^1}^2
\Big(\displaystyle\sup_x \|u_j(t)\|_{L^2(Q_1(x))}\Big)^{\frac{4}{N}},
\end{equation}
where $Q_a(x)$ denotes the square centered at $x$ whose edge has length $a$.
We proceed by contradiction. Assume that there exist a sequence $(t_n)$ of positive real numbers and $\epsilon >0$ such that $\displaystyle\lim_{n\to \infty}t_n =\infty$ and
\begin{equation} \label{IN}\|u_j(t_n)\|_{L^{2 + \frac{4}{N}}}>\epsilon\quad \mbox{for all}\; n\in \N. \end{equation}
By \eqref{GN} and \eqref{IN}, there exist a sequence $(x_n)$ in $\R^N$ and a positive real number denoted also by $\epsilon>0$ such that
\begin{equation}\label{IN1}\|u_j(t_n)\|_{L^2(Q_1(x_n))}\geq\epsilon,\quad \mbox{for all}\; n\in \N. \end{equation}
Let $\phi_j^n(x):=u_j(t_n,x +x_n).$ Using the conservation laws, we obtain
$$ \sup_n\|\phi_j^n\|_{H^1}<\infty.$$
Then, up to a subsequence extraction, there exists $\phi_j\in H^1$ such that $\phi_j^n$ convergence weakly to $\phi_j$ in $H^1.$ By Rellich Theorem, we have
$$ \displaystyle\lim_{n\to \infty}\|\phi_j^n - \phi_j\|_{L^2(Q_1(0))}=0.$$
Moreover, thanks to \eqref{IN1} we have, $\|\phi_j^n\|_{L^2(Q_1(0))}\geq \epsilon.$ So, we obtain
$$\|\phi_j\|_{L^2(Q_1(0))}\geq \epsilon.$$
We denote by ${\bf \bar{u}}:=(\bar{u}_1,..,\bar{u}_m)\in C(\R, H)$ the solution of \eqref{S} with data $\phi:=(\phi_1,..,\phi_m)$ and ${\bar u^n}:=({u}_1^n,..,u_m^n)\in C(\R, H)$ the solution of \eqref{S} with data $\phi^n:=(\phi_1^n,..,\phi^n_m).$ Take a cut-off function $\chi \in C_0^\infty(\R^N)$ which satisfies $0\leq \chi\leq1,\; \chi=1$ on $Q_1(0)$ and $supp(\chi)\subset Q_2(0).$ Using a continuity argument, there exists $T>0$ such that
$$\displaystyle\inf_{t\in[0, T]}\|\chi \bar{u}_j(t) \|_{L^2(\R^N)}\geq \frac{\epsilon}{2}.$$
Now, taking account of the claim \eqref{chi}, there is a positive time denoted also $T$ and $n_\epsilon\in \N$ such that $$\|\chi(u_j^n - \bar{u}_j)\|_{L_T^\infty(L^2)}\leq \frac{\epsilon}{4}\quad \mbox{for all}\; n\geq n_\epsilon.$$
Hence, for all $t\in [0, T]$ and $n\geq n_\epsilon,$
$$ \|\chi u_j^n(t)\|_{L^2}\geq \|\chi \bar{u}_j(t)\|_{L^2} - \|\chi(u_j^n - \bar{u}_j)(t)\|_{L^2}\geq \frac{\epsilon}{4}.$$
Using a uniqueness argument, it follows that $u^n_j(t,x)=u_j(t+t_n,x+x_n)$. Moreover, by the properties of $\chi$ and the last inequality, for all $t\in[0, T]$ and $n\geq n_\epsilon,$
$$ \|u_j(t+t_n)\|_{L^2(Q_2(x_n))}\geq \frac{\epsilon}{4}.$$
This implies that
$$\|u_j(t)\|_{L^2(Q_2(x_n))}\geq \frac{\epsilon}{4},\quad \mbox{for all}\; t\in [t_n, t_n + T]\;\mbox{and all}\; n\geq n_\epsilon.$$
Moreover, as $\displaystyle\lim_{n\to \infty}t_n=\infty,$
 we can suppose that $t_{n +1}- t_n>T$ for $n\geq n_\epsilon.$ Therefore, thanks to Morawetz estimates \eqref{mrwtz1}, we get for $N\geq4,$ the contradiction
\begin{eqnarray*}
1 &\gtrsim&\displaystyle\int_0^\infty\displaystyle\int_{\R^N\times\R^N}\frac{|u_j(t,x)|^2|u_j(t,y)|^2}{|x - y|^3}\,dxdydt\\
 &\gtrsim&\displaystyle\sum_n\displaystyle\int_{t_n}^{t_{n} +T}\displaystyle\int_{Q_2(x_n)\times Q_2(x_n)}|u_j(t,x)|^2|u_j(t,y)|^2\,dxdydt\\
&\gtrsim& \displaystyle\sum_nT\big(\frac{\epsilon}{4}\big)^4 = \infty.
\end{eqnarray*}
Using \eqref{mrwtz2}, for $N=3$, write
\begin{eqnarray*}
1
&\gtrsim&\int_0^{\infty}\|u_j(t)\|_{L^4(\R^3)}^4dt\\
&\gtrsim&\sum_n\int_{t_n}^{t_n+T}\|u_j(t)\|_{L^4(Q_2(x_n))}^4dt\\
&\gtrsim&\sum_n\int_{t_n}^{t_n+T}\|u_j(t)\|_{L^2(Q_2(x_n))}^4dt\\
&\gtrsim&\sum_n(\frac\varepsilon4)^4T=\infty.
\end{eqnarray*}
Using \eqref{mrwtz3}, for $N=2$, write
\begin{eqnarray*}
1
&\gtrsim&\int_0^{\infty}\|u_j(t)\|_{L^8(\R^2)}^4dt\\
&\gtrsim&\sum_n\int_{t_n}^{t_n+T}\|u_j(t)\|_{L^8(Q_2(x_n))}^4dt\\
&\gtrsim&\sum_n\int_{t_n}^{t_n+T}\|u_j(t)\|_{L^2(Q_2(x_n))}^4dt\\
&\gtrsim&\sum_n(\frac\varepsilon4)^4T=\infty.
\end{eqnarray*}
This completes the proof of Lemma \ref{t1}.\\
Finally, we are ready to prove scattering. By the two previous lemmas we have
$$ \|{\bf u}\|_{(S(t,\infty))^{(m)}}\lesssim \|\Psi\|_{H} + \epsilon(t) \|{\bf u}\|_{(S(t,\infty))^{(m)}}^{\frac{4p}{N(p-1)}-1},$$
where $ \epsilon(t)\to 0, \; \mbox{as}\; t\to \infty.$ It follows from Lemma \ref{Bootstrap} via $p_*<p<p^*$, that
$${\bf u} \in (S(\R))^{(m)}.$$
Now, let ${\bf v}(t)= T(-t){\bf u}(t).$ Taking account of Duhamel formula
$${\bf v}(t)= \Psi + i\displaystyle\sum_{k=1}^m\displaystyle\int_0^t T(-t)\big(|u_k|^p|u_1|^{p-2}u_1,..,|u_k|^p|u_m|^{p-2}u_m \big)\, ds.$$
Thanks to \eqref{sct1} and \eqref{sct3},
$$f_{j,k}({\bf u})\in L^{\frac{4p}{p(4-N)+N}}(\R, W^{1, \frac{2p}{2p -1}}),$$
so, applying Strichartz estimate, we get for $0<t<\tau,$
\begin{eqnarray*}
\|{\bf v}(t) - {\bf v}(\tau)\|_{H}
&\lesssim&\displaystyle\sum_{j,k=1}^m\big\|f_{j,k}({\bf u}) \big\|_{L^{\frac{4p}{p(4-N)+N}}((t,\tau), W^{1, \frac{2p}{2p -1}})}\stackrel{t,\tau\rightarrow\infty}{\longrightarrow0}.
\end{eqnarray*}
Taking ${\bf u_\pm}:=\lim_{t\rightarrow\pm\infty}{\bf v}(t)$, we get
$$\lim_{t\rightarrow\pm\infty}\|{\bf u}(t)-T(t){\bf u_{\pm}}\|_{H}=0.$$
Scattering is proved.
\end{proof}{}
\section{Appendix}
In what follows we give a classical proof, inspired by \cite{cgt,cks}, of Proposition \ref{prop2''} about Morawetz estimates. Let ${\bf u}:=(u_1,...,u_m)\in H$ be solution to
$$i\dot u_j +\Delta u_j= \displaystyle\sum_{k=1}^{m}a_{jk}|u_k|^p|u_j|^{p-2}u_j $$
in $N_1$-spatial dimensions and ${\bf v}:=(v_1,...,v_m)\in H$ be solution to
$$i\dot v_j +\Delta v_j+ \displaystyle\sum_{k=1}^{m}a_{jk}|v_k|^p|v_j|^{p-2}v_j =0$$
in $N_2$-spatial dimensions. Define the tensor product ${\bf w}:= ({\bf u}\otimes{\bf v})(t,z)$ for $z$ in
$$  \R^{N_1 +N_2}:= \{ (x,y)\quad\mbox{s. t}\quad x\in \R^{N_1}, y\in \R^{N_2}\}$$
by the formula
$$ ({\bf u}\otimes{\bf v})(t,z) = {\bf u}(t,x){\bf v}(t,y) .$$
Denote $F({\bf u}):= \displaystyle\sum_{k=1}^{m}a_{jk}|u_k|^p|u_j|^{p-2}u_j.$ A direct computation shows that ${\bf w}:=(w_1,...,w_n)= {\bf u}\otimes{\bf v}$ solves the equation
\begin{equation}\label{tensor1}
i\dot  w_j +\Delta w_j- F({\bf u})\otimes v_j -  F({\bf v})\otimes u_j:=i\dot  w_j +\Delta w_j+ h=0
\end{equation}
where $\Delta:= \Delta_x + \Delta_y.$ Define the Morawetz action corresponding to ${\bf w}$ by
\begin{eqnarray*}
M_a^{\otimes_2}
&:=& 2\displaystyle\sum_{j=1}^m\displaystyle\int_{\R^{N_1}\times \R^{N_2}}\nabla a(z).\Im(\overline{u_j\otimes v_j(z)}\nabla (u_j\otimes v_j)(z))\,dz.
\end{eqnarray*}
where $\nabla: =(\nabla_x,\nabla_y).$ It follows from the equation \eqref{tensor1} that
\begin{gather*}
 \Im(\dot{  \bar{w}}_j\partial_i w_j) =\Re (-i\dot{  \bar{w}}_j\partial_i w_j)= - \Re \big((\Delta \bar{w}_j +\displaystyle\sum_{k=1}^{m}a_{jk}|\bar{u}_k|^p|\bar{u}_j|^{p-2}\bar{u}_j \bar{v}_j +\displaystyle\sum_{k=1}^{m}a_{jk}|\bar{v}_k|^p|\bar{v}_j|^{p-2}\bar{v}_j \bar{u}_j)\partial_i w_j\big);\\
 \Im( \bar{w}_j\partial_i\dot  w_j) =\Re (-i \bar{w}_j\partial_i\dot  w_j)=\Re \big(\partial_i(\Delta w_j +\displaystyle\sum_{k=1}^{m}a_{jk}|u_k|^p|u_j|^{p-2}u_j v_j +\displaystyle\sum_{k=1}^{m}a_{jk}|v_k|^p|v_j|^{p-2}v_j u_j) \bar{w}_j\big).
\end{gather*}
Moreover, denoting the quantity $ \big\{ h,w_j\big\}_p:=\Re \big(h\nabla\bar{w}_j - w_j\nabla\bar{h} \big)$, we compute
\begin{eqnarray*}
 \big\{ h,w_j\big\}_p^i
& = &\Re\Big[\partial_i\Big(\displaystyle\sum_{k=1}^{m}a_{jk}|\bar{u}_k|^p|\bar{u}_j|^{p-2}\bar{u}_j \bar{v}_j +\displaystyle\sum_{k=1}^{m}a_{jk}|\bar{v}_k|^p|\bar{v}_j|^{p-2}\bar{v}_j \bar{u}_j\Big) w_j\\
& -& \Big(\displaystyle\sum_{k=1}^{m}a_{jk}|u_k|^p|u_j|^{p-2}u_j v_j +\displaystyle\sum_{k=1}^{m}a_{jk}|v_k|^p|v_j|^{p-2}v_j u_j\Big) \partial_i\bar{w}_j\Big].
\end{eqnarray*}
Using the identity
$$\partial_k\Big(\partial^2_{i,k}(|w_j|^2)-4\Re(\partial_i\bar{w}_j\partial_kw_j)\Big)=2\Re\Big(\bar{w}_j\partial_i \Delta w_j - \partial_iw_j \Delta\bar{w}_j\Big),$$
it follows that, for any convex function $a$ 
\begin{eqnarray*}
\dot  M_a^{\otimes_2}
&=& 2\displaystyle\sum_{j=1}^m\displaystyle\int_{\R^{N_1}\times \R^{N_2}}\partial_i a\Re \big(\bar{w}_j\partial_i \Delta w_j - \partial_iw_j \Delta\bar{w}_j\big)\,dz  - 2\displaystyle\sum_{j=1}^m \displaystyle\int_{\R^{N_1}\times \R^{N_2}}\partial_i a \big\{h,w_j\big\}_p^i\,dz\\
&=&\displaystyle\sum_{j=1}^m\displaystyle\int_{\R^{N_1}\times \R^{N_2}}\partial_i a\partial_k\Big(\partial^2_{i,k}(|w_j|^2)-4\Re(\partial_i\bar{w}_j\partial_kw_j)\Big)\,dz  - 2\displaystyle\sum_{j=1}^m \displaystyle\int_{\R^{N_1}\times \R^{N_2}}\partial_i a \big\{h,w_j\big\}_p^i\,dz\\
&=&-\displaystyle\sum_{j=1}^m\displaystyle\int_{\R^{N_1}\times \R^{N_2}}(\Delta^2 a)|w_j|^2-4\Re(\partial^2_{i,k}a\partial_i\bar{w}_j\partial_kw_j)\Big)\,dz  - 2\displaystyle\sum_{j=1}^m \displaystyle\int_{\R^{N_1}\times \R^{N_2}}\partial_i a \big\{h,w_j\big\}_p^i\,dz\\
&\geq&-\displaystyle\sum_{j=1}^m\displaystyle\int_{\R^{N_1}\times \R^{N_2}}(\Delta^2 a)|w_j|^2\,dz  - 2\displaystyle\sum_{j=1}^m \displaystyle\int_{\R^{N_1}\times \R^{N_2}}\partial_i a \big\{h,w_j\big\}_p^i\,dz.
\end{eqnarray*}
Then
\begin{eqnarray*}
\sup_{[0,T]}|M_a^{\otimes_2}|
&\geq& \displaystyle\sum_{j=1}^m \displaystyle\int_0^T \displaystyle\int_{\R^{N_1}\times \R^{N_2}}\Big((-\Delta^2a)|u_jv_j|^2 +4(1 - \frac{1}{p})\Delta_x a\displaystyle\sum_{k=1}^ma_{jk}|u_k|^p|u_j|^p|v_j|^2 \\
&+& 4(1 - \frac{1}{p})\Delta_y a\displaystyle\sum_{k=1}^ma_{jk}|v_k|^p|v_j|^p|u_j|^2 \Big)\,dz\,dt\\
&\geq& \displaystyle\sum_{j=1}^m \displaystyle\int_0^T \displaystyle\int_{\R^{N_1}\times \R^{N_2}}(-\Delta^2a)|u_jv_j|^2\,dz\,dt.
\end{eqnarray*}
Now, take $a(z):=a(x,y) = |x - y|$, where $(x,y)\in\R^{N}\times \R^{N}.$ Then, direct calculation yields
$$-\Delta^2 a =\left\{\begin{array}{ll}
C_1\delta(x-y),&\mbox{if}\quad N=3;\\
C_2|x-y|^{-3},&\mbox{if}\quad N\geq4.
\end{array}\right.$$
When $N =3$, choosing $u_j = v_j,$ we get
$$ \displaystyle\sum_{j=1}^m \displaystyle\int_0^T \displaystyle\int_{\R^3} |u_j(t,x)|^4\,dx\,dt\lesssim \displaystyle\sup_{[0,T]}|M_a^{\otimes_2}|.$$
If $N>3$, it follows that
$$ \displaystyle\sum_{j=1}^m \displaystyle\int_0^T\displaystyle\int_{\R^{N}\otimes \R^{N}}\frac{|u_j(t,x)|^2|u_j(t,y)|^2}{|x - y|^3}\,dx\,dy\,dt\lesssim \displaystyle\sup_{[0,T]}|M_a^{\otimes_2}|. $$
For $N=2$, the proof follows by taking $a(x,y)$ like in \cite{cgt}.\\



\begin{thebibliography}{99}


\bibitem{AC1}{\bf D. R. Adams}: {\em Sobolev Spaces}. Academic Press, New York, (1975).

\bibitem{AC2}{\bf A. Ambrosetti and E. Colorado}, {\em Bound and ground states of coupled nonlinear Schr\"odinger
equations}, C. R. Math. Acad. Sci. Paris, 342, 453-458, (2006).

\bibitem{AC3}{\bf T. Bartsch and Z.-Q. Wang}, {\em Note on ground states of nonlinear Schr\"odinger systems}, J. Partial
Differential Equations, 19, 200-207, (2006).

\bibitem{AC4}{\bf T. Bartsch, Z.-Q. Wang and J. Wei}, {\em Bound states for a coupled Schr\"odinger system}, J. Fixed
Point Theory Appl., 2, 353-367, (2007).

\bibitem{J.B1}{\bf J. Bourgain}: {\it Global well-posedness of defocusing critical nonlinear Schr\"odinger equation in the radial case}, J. Amer. Math. Soc. 12, No. 1, 145-171, (1991).

\bibitem{J.B2}{\bf J. Bourgain}: {\it Global solutions of nonlinear Schr\"odinger equation}, American Mathematical Society Colloquium Publications, 46. American Mathematical Society, Providence, RI, (1991).

\bibitem{Cas}{\bf T. Cazenave}, {\em An introduction to nonlinear Schr\"odinger equations}, Textos de Metodos Matematicos {26}, Instituto de Matematica UFRJ, (1996).

\bibitem{Cas.F}{\bf T. Cazenave and F. B. Weissler}, {\it Critical nonlinear Schr\"odinger equation}, Non. Anal. TMA, {14}, 807-836, (1990).

\bibitem{cgt}{\bf J. Colliander, M. Grillakis and N. Tzirakis}, {\em Tensor products and correlation estimates with applications to nonlinear Schr\"odinger equations}, Communications on pure and applied mathematics, 62, 920-968, (2009).

\bibitem{Col.K}{\bf J. Colliander, M. Keel, G. Staffilani, H. Takaoka and T. Tao}: {\it Global existence and scttering for rough solutions of a nonlinear Schr\"odinger equation on $\R^3$}, Communications on Pure and Applied Mathematics, 57(8),987-1014, (2004).

\bibitem{cks}{\bf J. Colliander, M. Keel, G. Staffilani, H. Takaoka and T. Tao}, {\it Global existence and scttering for the energy critical nonlinear Schr\"odinger equation in $\R^3$}, Annals of Mathematics, 167, 767-865, (2008).


\bibitem{Hasegawa}{\bf A. Hasegawa and F. Tappert},{\em Transmission of stationary nonlinear optical pulses in dispersive dielectric fibers II. Normal dispersion}, Appl. Phys. Lett. 23, 171-172, (1973).

\bibitem{hs}{\bf F. T. Hioe and T. S. Salter}, {\em Special set and solutions of coupled nonlinear Sch\"odinger
equations}, J. Phys. A: Math. Gen., 35, 8913-8928, (2002).

\bibitem{mz}{\bf L. Ma and L. Zhao}, {\em Sharp thresholds of blow-up and global existence for the coupled nonlinear Schr\"odinger system}, J. Math. Phys. 49, 062103, (2008).

\bibitem{ntds}{\bf Nghiem V. Nguyen, Rushun Tian, Bernard Deconinck and Natalie Sheils}, {\em Global existence for a coupled system of Schr\"odinger equations with power-type nonlinearities}, J. Math. Phys. 54, 011503 (2013).


\bibitem{RV07}{\bf E. Ryckman and M. Visan}: {\it Global well-posedness and scattering for the defocusing energy-critical nonlinear Schr\"odinger equation in $\R^{1+4}$},  Amer. J. Math,  129, No. 1, 1-60, (2007).

\bibitem{T}{\bf T. Saanouni}, {\it Global well-posedness and scattering of a $2D$ Schr\"odinger equation with exponential growth}, Bull. Belg. Math. Soc. Simon Stevin, 17, 441-462, (2010).


\bibitem{T3}{\bf T. Saanouni}: {\em Remarks on the semilinear Schr\"odinger equation}, J. Math. Anal. Appl. 400, 331-344, (2013).

\bibitem{T1}{\bf T. Saanouni}: {\it Scattering for a $2d$ Schr\"odinger equation with exponential growth in the conformal space}, Math. Method. App. Sci, 33, No. 8, 1046-1058, (2010).


\bibitem{T2}{\bf T. Saanouni}, {\it A note on coupled nonlinear Schrödinger equations}, Advances in Nonlinear Analysis,  3, No. 4, 247-269, (2014).


\bibitem{xs}{\bf Xianfa Song}, {\em Stability and instability of standing waves to a system of Schr\"odinger
equations with combined power-type nonlinearities}, J. Math. Anal. Appl. 366, 345-359, (2010).

\bibitem{Tao} {\bf T. Tao}, {\em Nonlinear dispersive equations: local and global analysis}, CBMS Regional Series in Mathematics, 106 (American Mathematical Society, Providence, RI, 2006).
		


\bibitem{V}{\bf M. Visan}, {\it The defocusing energy-critical nonlinear Schr\"odinger equation in higher dimensions}, Duke Math. J. 138, no. 2,   281-374, (2007).

\bibitem{w}{\bf Y. Xu}, {\em Global well-posedness, scattering, and blowup for nonlinear coupled Schrödinger equations in $\R^3$}, to appear in Applicable Analysis.

\bibitem{Zakharov} { \bf V. E. Zakharov}, {\em Stability of periodic waves of finite amplitude on the surface of a deep fluid}. Sov. Phys. J. Appl. Mech. Tech. Phys. 4, 190-194, (1968).



\end{thebibliography}
\end{document}